\documentclass[epic,eepic,11pt]{amsart}

\usepackage{amsfonts,mathrsfs}
\usepackage[mathscr]{eucal}
\usepackage{mathtools}
\usepackage{amssymb}
\usepackage{amscd}
\usepackage{fancyhdr}
\usepackage[all,cmtip]{xy}

 \usepackage[all]{xy}
\usepackage{tikz-cd}

\usepackage{euler,eucal}

\DeclareMathAlphabet{\mathbf}{T1}{ppl}{bx}{n}
\DeclareMathAlphabet{\mathrm}{T1}{ppl}{m}{n}

\numberwithin{equation}{section}

\newcommand\note[1]%
{$^\dagger$\marginpar{\footnotesize{$^\dagger${#1}}}}

\def\({\left(}
\def\){\right)}
\def\<{\left<}
\def\>{\right>}

\newtheorem{theorem}{Theorem}[section]
\newtheorem{proposition}[theorem]{Proposition}
\newtheorem{lemma}[theorem]{Lemma}
\newtheorem{definition}[theorem]{Definition}

\newtheorem{corollary}[theorem]{Corollary}

\theoremstyle{definition}

\newtheorem{remark}[theorem]{Remark}

\newcommand     {\comment}[1]   {}
\newcommand{\mute}[2] {}
\newcommand     {\printname}[1] {}

\newcommand\funclim[1]{\operatorname*{\mathrm{#1}}}

\renewcommand\lim{\funclim{lim}}

\newcommand\sur{\mathrel{\to\kern-1.8ex\to}}
\newcommand\iso{\mathrel{\hookrightarrow\kern-1.8ex\to}}

\newcommand\longhookrightarrow{\lhook\joinrel\longrightarrow}

\newcommand\longsur{\mathrel{\longrightarrow\kern-1.8ex\to}}
\newcommand\longiso{\mathrel{\longhookrightarrow\kern-1.8ex\to}}

\begin{document}

\title{Kirwan surjectivity for the equivariant Dolbeault cohomology}

\author{Yi Lin}

\address{Yi Lin \\ Department of Mathematical Sciences \\  Georgia Southern University \\Statesboro, GA, 30460 USA}

\email{yilin@georgiasouthern.edu}

\date{\today}



\begin{abstract}   Consider the holomorphic Hamiltonian action of a compact Lie group $K$ on a compact K\"ahler manifold $M$ with a moment map $\Phi: M\rightarrow \mathfrak{k}^*$. Assume that $0$ is a regular value of the moment map. Weitsman raised the question of what we can say about the cohomology of the K\"ahler quotient $M_0:=\Phi^{-1}(0)/K$ if all the ordinary cohomology of $M$ is of type $(p, p)$.  

In this paper, using the Cartan-Chern-Weil theory we show that in the above context there is a natural surjective Kirwan map from an equivariant version of the Dolbeault cohomology of $M$ onto the Dolbeault cohomology of the K\"ahler quotient $M_0$. As an immediate consequence, this result provides an answer to the question posed by Weitsman.

\end{abstract}

\maketitle

\section{Introduction}
Assume that there is a compact Lie group $K$ acting on a compact symplectic manifold $(M,\omega)$ in a Hamiltonian fashion with a moment map $\Phi: M\rightarrow \mathfrak{k}^*$,  and that $K$ acts freely on the level set $Z:=\Phi^{-1}(0)$. Then the quotient space $M_0:=Z/K$ naturally inherits a symplectic structure from that of $M$, and is called a symplectic quotient of the Hamiltonian $K$-manifold $M$. In her fundamental work \cite{Kir84}, Kirwan established the important Kirwan surjectivity theorem for compact Hamiltonian $K$-manifolds, which asserts that the Kirwan map $\kappa: H_K(M)\rightarrow H(M_0)$ is surjective.

Now assume that the Hamiltonian $K$-manifold $(M,\omega)$ is equivariant K\"ahler. In other words, assume that the symplectic $2$-form $\omega$ is K\"ahler, and that the action of $K$ is holomorphic. Then the symplectic quotient $M_0$ inherits a K\"ahler structure from that of $M$, and is called a K\"ahler quotient of the equivariant K\"ahler $K$-manifold $M$, c.f. \cite{GS82}. It is well known that in this case the action of $K$ naturally extends to an action of $G:=K^{\mathbb{C}}$. Moreover, when $M$ is a non-singular projective variety, and when the action of $G$ is linear,  the famous Kempf-Ness theorem \cite{KN79} asserts that the K\"ahler quotient $M_0$ can be naturally identified with the GIT quotient of $M$ by $G$.

 Kirwan studied \cite[Sec. 14]{Kir84} the linear action of a reductive algebraic group $G$ on a non-singular projective variety $M$ from the view point of GIT quotient. She showed that there is a Hodge structure on $H_G(M, \mathbb{Q})$, and that the surjective Kirwan map $H_G(M,\mathbb{Q})\rightarrow H(M_0,\mathbb{Q})$ in this case is strictly compatible with the Hodge structures.  In particular, this implies that if the Hodge numbers of $M$ satisfy $h^{p,q}=0$ when $p\neq q$, then the same is true for $M_0$.  However, Kirwan's method is algebro-geometric, and does not apply to the more general case of equivariant K\"ahler manifolds. Indeed, Weitsman raised the following question in an AIM moment map geometry workshop  \cite{AIM04} that took place in August 2004.
  
 \textbf{Question}: Suppose $M$ is a compact K\"ahler manifold and a Hamiltonian $K$-space. Suppose all the ordinary cohomology is of type $(p, p)$. Can we say anything about the cohomology of the quotient?
 
 On a complex manifold, it is well known that if the $\overline{\partial}\partial$-lemma holds, then there is a pure Hodge structure on its ordinary cohomology.  On an equivariant K\"ahler 
 Hamiltonian $K$-manifold $M$ , Lillywhite (\cite{Lilly98}, \cite{Lilly03}) and Teleman \cite{T00}  showed that the $\overline{\partial}_K\partial_K$-lemma holds on the Cartan complex $\Omega_K(M, \mathbb{C})$ of equivariant differential forms, which in particular implies that there is a pure Hodge structure on the equivariant De Rham cohomology $H_K(M, \mathbb{C})$. Let $M_0$ be the K\"ahler quotient of the equivariant K\"ahler Hamiltonian $K$-manifold $M$ taken at the zero level set. In this paper we show that the Kirwan map $\kappa: H_K(M,\mathbb{C})\rightarrow H(M_0,\mathbb{C})$  respects the Hodge structures on $H_K(M,\mathbb{C})$ and on $H(M_0,\mathbb{C})$ respectively. As an immediate consequence,  this implies that if the Hodge numbers of $M$ are concentrated on the diagonal, then the same property holds for $M_0$.

Our method relies on the Cartan-Chern-Weil theory in differential geometry, which we briefly explain here. First note that the zero level set $Z:=\Phi^{-1}(0)$ is the total space of a principal $K$-bundle $\pi: Z\rightarrow M_0$. In this setup, there is a Cartan operator
 $\mathcal{C}: \Omega_K(Z)\rightarrow \Omega(M_0)$ in the equivariant De Rham theory,  which is a homotopy equivalence from the Cartan complex $\{\Omega_K(Z), d_K\}$ to the De Rham complex $\{\Omega(M_0), d\}$. The definition of the Cartan operator depends on the choice of a connection on $Z$. However, in our situation the restriction of the K\"ahler metric to $Z$ provides a canonical connection, and thus a canonical Cartan operator $\mathcal{C}$. As a result, we have the following commutative diagram.
 
  \[  \begin{tikzcd}
   \Omega_K(M,\mathbb{C})    \arrow[rd, "\kappa"] \arrow[r, "i^*"]       &  \Omega_K(Z, \mathbb{C})\arrow[d, "\mathcal{C}"] \\ &\Omega(M_0,\mathbb{C})
   \end{tikzcd}\]
  
  Here the top horizontal map $i^*$ is induced by the inclusion map $i: Z\rightarrow M$, and the vertical map is the Cartan operator defined using the canonical connection on $Z$.  We define the diagonal map to be the canonical Kirwan map at the level of differential forms.

 We observe that the action of $K$ induces a transversely K\"ahler foliation on $Z$, and that with respect to this transverse K\"ahler structure, the curvature $2$-forms associated to the canonical connection on $Z$ are horizontal forms of type $(1, 1)$.  As a consequence, the Kirwan map $\kappa:\Omega_K(M)\rightarrow \Omega(M_0)$ respects the bi-gradings on
  $\Omega_K(M)$ and $\Omega(M_0)$ induced by the complex structures on $M$ and by the quotient complex structure on $M_0$ respectively. This enables us to show that the usual Kirwan map at the level of cohomologies is a morphism of Hodge structures, and to show that there is a surjective Kirwan map from the equivariant Dolbeault cohomology of $M$ onto the Dolbeault cohomology of $M_0$.

 We would like to point out that technically we only assume that the $K$ action on the zero level set $Z$ is locally free, so that our result applies to the general case when the K\"ahler quotient $M_0$ is an orbifold.  We note that to show there is a quotient K\"ahler structure on $M_0$, it is equivalent to show that the foliation on $Z$ induced by the locally free $K$-action is transversely K\"ahler; moreover, the Hodge structure on $H(M_0)$ is naturally isomorphic to the Hodge structure on $H_B(Z)$, the basic cohomology of $Z$ as a foliated manifold.  Thus to establish the main result of this paper it suffices to show that the Kirwan map is a morphism of Hodge structures from $H_K(M)$ to $H_B(Z)$.  It has been long known that the cohomology of a K\"ahler orbifold exhibits properties very similar to that of a K\"ahler manifold. However, it has been difficult to find a self-contained reference in the literature until more recently \cite{BBFMT17} appeared. Our foliation approach offers an alternative self-contained simple treatment of Hodge theory on a generic K\"ahler quotient. 
  
 This paper is organized as follows. Section \ref{Cartan-chain-homo} reviews equivariant De Rham cohomology theory, especially the definition of the Cartan operator. Section \ref{homological-algebra} explains for an abstract double complex how the $dd'$-lemma would lead to a Hodge type decomposition. Section \ref{eq-dolbeault} reviews the $\overline{\partial}_K\partial_K$-lemma for an equivariant K\"ahler manifold. Section \ref{foliation} presents some background materials on transversely K\"ahler foliations.  Section \ref{hodge-structure} proves that on an equivariant K\"ahler manifold the Kirwan map is a morphism of Hodge structures.

\section{Review of Equivariant cohomology theory}\label{Cartan-chain-homo}

We begin with a rapid review of equivariant de Rham theory and refer
to \cite{GS99} for a detailed account. Let $K$ be a compact Lie group with Lie algebra $\mathfrak{k}$, 
and let $\Omega_K(M) = (S\mathfrak{k}^*\otimes \Omega(M))^K$ be the Cartan complex of the $K$-manifold
$M$.  By definition
an element of $\Omega_K(M)$ is an equivariant polynomial from $\mathfrak{k}$ to $\Omega(M)$ and is called an
\emph{equivariant differential form} on $M$. The bigrading of the Cartan complex
is defined by $\Omega_K^{ij}(M) = (S^i\mathfrak{k}^*\otimes \Omega^{j-i}(M))^K$. It is equipped with the vertical differential 
$1 \otimes d$, which we will abbreviate to $d$, and the horizontal differential $d'$, which is defined by $d'\alpha(\xi) = -\iota(\xi)\alpha(\xi)$. Here $\iota(\xi)$ denotes inner product with the vector field on $M$ induced by $\xi \in \mathfrak{k}$. As a single complex, $\Omega_K(M)$ has the grading $\Omega^r_K(M) = \bigoplus_{i+j=r}\Omega_K^{ij}(M)$ and the total differential $d_K = d + d'$, which is called the \emph{equivariant exterior derivative}. The total cohomology $\text{ker}d_K/\text{im}d_K$ is the \emph{equivariant De Rham cohomology} $H_K(M)$.

We say that a form $\gamma\in \Omega(M)$ is \emph{horizontal}, if for all $\xi\in \mathfrak{k}$, $\iota(\xi)\gamma=0$.  We say that a form $\gamma\in \Omega(M)$ is \emph{basic}, if it is horizontal, and if for all $\xi\in\mathfrak{k}$, $\mathcal{L}(\xi)\alpha=0$. We will denote by
$\Omega_{\text{hor}}(M)$ the space of horizontal forms on $M$, and by $\Omega_{bas}(M)$ the space of basic forms on $M$. 
By definition, for all $1\leq l\leq k$, the curvature $2$-form $\mu^l$ is horizontal. It is also clear that $d\gamma\in \Omega_{bas}(M)$ for all $\gamma\in \Omega_{bas}(M)$. Thus we have a differential complex of basic forms $\{\Omega_{bas}(M), d\}$, which is a subcomplex of the usual de Rham complex on $M$. Its cohomology is called the \emph{basic cohomology}, and is denoted by $H_B(M)$.

Now suppose that $K$ acts locally freely on $M$. Let $\xi_1,\cdots, \xi_k$ be a basis of $\mathfrak{k}$, let $x^1,\cdots, x^k$ be the corresponding coordinates, i.e., the corresponding dual basis in $\mathfrak{k}^*$, and let $c_{ij}^l$'s be the structure constants of the Lie algebra $\mathfrak{k}$ relative to this basis. Then there exist $1$-forms $\theta^1,\cdots,\theta^k$, called the \emph{connection $1$-forms}, such that
\begin{equation}\label{connection-elements} \iota(\xi_i)\theta^j=\delta_i^j,\,\,\,\mathcal{L}(\xi_i)\theta^j=-c^j_{il}\theta^l,\end{equation}
where $\mathcal{L}(\xi_i)$ denotes the Lie derivative with the vector field on $M$ induced by $\xi_i\in \mathfrak{k}$. 
In this context, for all $1\leq l\leq k$, the curvature $2$-form $\mu^l$ is given by
\begin{equation}\label{curvature-2-form} \mu^l=d\theta^l+\frac{1}{2} c_{ij}^l\theta^i\theta^j.\end{equation}

For each multi-index
\begin{equation}\label{multi-index}I:= (i_1,\cdots, i_r), \, 1\leq i_1<\cdots <i_r\leq k,\end{equation} let 
\[ \theta^I=\theta^{i_1}\cdots \theta^{i_r},\,x^I=x^{i_1}\cdots x^{i_r}, \, \mu^I=\mu^{i_1}\cdots \mu^{i_r}\] denote the corresponding monomials in $\theta^j$, $x^j$ and $\mu^j$ respectively.


\begin{theorem}\label{horizontal-decom} (\cite[Thm 3.4.1]{GS99})
Suppose that the action of $K$ on $M$ is locally free, and that $\theta^1,\cdots,\theta^k$ are connection $1$-forms satisfying (\ref{connection-elements}).  Then every differential form $\alpha\in \Omega(M)$
can be written uniquely as
\[ \alpha =\displaystyle \sum_I\theta^I h_I,\]
where  $h_I\in \Omega_{\text{hor}}(M)$ for each multi-index $I$ as given in (\ref{multi-index}).\end{theorem}

Theorem \ref{horizontal-decom} implies immediately that there is a projection operator \begin{equation}\label{hor-proj} Hor: \Omega(M)\rightarrow \Omega_{\text{hor}}(M).\end{equation} It further gives rise to the following projection operator on the Cartan complex.
\begin{equation}\label{proj} Hor: \Omega_K(M) =\left(S\mathfrak{k}^*\otimes\Omega(M)\right)^K\rightarrow \left(S\mathfrak{k}^*\otimes\Omega_{\text{hor}}(M)\right)^K.\end{equation}

\begin{definition}\label{cartan-operator} The Cartan operator 
\begin{equation}\label{cartan} \mathcal{C}: \Omega_K(M)\rightarrow \Omega_{bas}(M)\end{equation}
is the composition of the projection operator 
(\ref{proj}) and the map
\[ (S\mathfrak{k}^*\otimes \Omega_{\text{hor}}(M))^K\rightarrow \Omega_{bas}(M)\]
coming from the ``evaluation map'' 
\[ x^I\otimes \alpha\mapsto \mu^I\alpha.\]
\end{definition}

\begin{remark} \label{dependence} Set $\theta=\sum_{i=1}^k\theta^i\otimes \xi_i$. Then $\theta$ is a $\mathfrak{k}$-valued connection one form satisfying
\begin{equation}\label{connection-form} \iota(\xi)\theta=\xi, \,\forall\,\xi\in\mathfrak{k}, \, (R_h)^*\theta=ad(h^{-1})\theta,\,\forall\, h\in K.\end{equation} Here $ad$ denotes the adjoint representation of $K$ on $\mathfrak{k}$. For $ x\in M$, let $V_x$ be the tangent space to the group orbit $K\cdot x$ at $x$, and let $H_x$ be the kernel of the one form $\theta_x$. Then we get a connection on $M$, i.e.,  a $K$-invariant splitting
\begin{equation}\label{splitting} T_x(M)=V_x\oplus H_x, \, x\in M.\end{equation}
Conversely, any $K$-invariant splitting as given in (\ref{splitting}) determines a $\mathfrak{k}$-valued connection one form $\theta$ satisfying (\ref{connection-form}). A close inspection shows that the definition of the Cartan map $\mathcal{C}$ does not depend on the choice of a basis $\xi_1,\cdots,\xi_k$ in $\mathfrak{k}$. It depends only on the connection given in (\ref{splitting}).

\end{remark}

\begin{theorem}\label{cartan-homotopy} (\cite[Thm 5.2.1]{GS99}) The Cartan operator (\ref{cartan}) is a chain homotopy equivalence from the Cartan complex $\{\Omega_K(M), d_K\}$ to the complex of basic forms $\{\Omega_{bas}(M),d\}$. 
\end{theorem}

We finish this section by recalling the definition of Hamiltonian symplectic manifolds, the Kirwan-Ginzburg equivariant formality theorem, and the Kirwan surjectivity theorem.

\begin{definition} Consider the action of a Lie group $K$ on a symplectic manifold $(M,\omega)$. Let $\mathfrak{k}^*$ be the dual of the Lie algebra $\mathfrak{k}$ of $K$. We say that the action of $K$ is Hamiltonian, if there exists an equivariant map $\Phi: M\rightarrow \mathfrak{k}^*$, called the moment map, satisfying the Hamiltonian equation:
\begin{equation}\label{Hamiltonian-eq} -\iota(\xi)\omega= d< \Phi, \xi>, \,\,\forall\,\xi\in\mathfrak{k},\end{equation}
where $<\cdot, \cdot>$ denotes the natural pairing between $\mathfrak{k}$ and $\mathfrak{k}^*$. 

\end{definition}

\begin{theorem}\label{formality}(\textbf{Kirwan-Ginzburg Equivariant formality theorem}) (\cite{Kir84}, \cite{Gin87}) Suppose that the action of a compact Lie group $K$ on a compact symplectic manifold $(M, \omega)$ is Hamiltonian. Then the Hamiltonian $K$-manifold $M$ is equivariantly formal, i.e., there is an isomorphism of $(S\mathfrak{t}^*)^K$-modules
\[H_K(M)\cong (S\mathfrak{k}^*)^K\otimes H(M).\]
\end{theorem}

\begin{theorem}\label{Kirwan-surj}(\textbf{Kirwan surjectivity theorem}) (\cite{Kir84}) Consider the Hamiltonian action of a compact Lie group $K$ on a compact symplectic manifold $(M,\omega)$ with a moment map $\Phi: M\rightarrow \mathfrak{k}^*$. Assume that $0\in \mathfrak{k}^*$ is a regular value.
Then the Kirwan map $ \kappa: H_K(M)\rightarrow H_K(\Phi^{-1}(0))$ induced by the inclusion map
$i: \Phi^{-1}(0)\rightarrow M$ is surjective.
\end{theorem}

\section{ The $dd'$-lemma for an abstract double complex}\label{homological-algebra}

Let $(K^{**}, d, d')$ be a double complex of $R$-modules (over a commutative ring $R$). In other words, let $K^{**}$ be a doubly graded complex of $R$-modules equipped with a horizontal differential $d'$ and a vertical differential $d$ which are both of degree one, and which anti-commute with each other. Throughout this section $(K^{**}, d, d')$ is assumed to be bounded in the following sense: for each $n$, there are only finitely many non-zero components in the direct sum $K^n = \bigoplus_{i+j=n} K^{i,j}$. Set $D=d+d'$, and $K_{d'}=\text{ker}\,{d'}\cap K$. Since $d$ anti-commutes with $d'$, $\{K_{d'}^{p, *}, d\}$ is a differential complex for all integer $p$. In what follows, we will denote by $H^{p,q}(K, d)$ and $H^{p,q}(K_{d'}, d)$ the $q$-th cohomology associated to the differential complexes $\{K^{p,*}, d\}$ and $\{K^{p,*}_{d'}, d\}$ respectively.

It is easy to see that the inclusion map $K_{d'}\hookrightarrow K$ induces two morphisms of differential complexes as follows.
\begin{equation}\label{morphism1}  \{K_{d'}, d\}\rightarrow \{K, d\},
\end{equation}
\begin{equation}\label{morphism2}  \{K_{d'}, d\}\rightarrow \{K, D\}.
\end{equation}
 Clearly, (\ref{morphism1}) and (\ref{morphism2})  further induce two homomorphisms of cohomologies respectively as follows.
\begin{equation}\label{homo1} H^{p,q}(K_{d'}, d)\rightarrow H^{p,q}(K, d),\end{equation}   
\begin{equation}\label{homo2} \displaystyle \bigoplus_{p+q=r}H^{p,q}(K_{d'}, d)\rightarrow H^r(K, D). \end{equation}
\begin{proposition}\label{decomposition}  Assume that the $dd'$-lemma holds for the double complex $(K^{**}, d, d')$. That is to say that
\begin{equation}\label{dd'-lemma} \text{ker}\,d \cap \text{im} d'=\text{im}\, d'\cap \text{ker}\, d=\text{im}\, dd'
\end{equation}
Then the following properties hold true.
\begin{itemize} \item [a)] The homomorphism (\ref{homo1}) induced by (\ref{morphism1}) is an isomorphism.
\item[b)] The homomorphism (\ref{homo2}) induced by (\ref{morphism2}) is an isomorphism. \end{itemize}
\end{proposition}

\begin{proof} 


\begin{itemize}

\item[a)]  We first show that the map (\ref{homo1}) is injective. Suppose that the image of $[\alpha]\in  H^{p,q}(K_{d'},d)$ under (\ref{homo1}) is zero, where $\alpha\in K^{p,q}_{d'}$. Then there exists
$\beta \in K^{p,q-1}$ such that $\alpha=d\beta$. Thus $\alpha$ is both $d$-exact
and $d'$-closed. Therefore by the $dd'$-lemma there exists $\gamma \in K^{p,q}$ such that
$\alpha=dd' \gamma$. It follows that $\alpha$ must represent a trivial cohomology class in $ H^{p,q}(K_{d'},d)$.

Next we show that the map (\ref{homo1}) is surjective. Suppose that $[\alpha]\in H^{p,q}(K, d)$, where $\alpha\in K^{p,q}$. Then we have that $d\alpha=0$. Since $d$ anti-commutes with $d'$, $d'\alpha$ is both $d'$-exact and 
$d$-closed.  Thus by the $dd'$-lemma there exists $\beta\in K^{p,q}$ such that
$d'\alpha=d'd\beta$.  Clearly, $\alpha-d\beta \in K_{d'}$ represents a class in 
$H^{p,q}(K_{d'},d)$ whose image under (\ref{iso-1}) is $[\alpha-d\beta]=[\alpha]$. This completes the proof.

\item[b)] We first note that the injectivity of (\ref{homo2}) follows directly from \cite[Lemma 5.15]{DGMS75}. It suffices to show that (\ref{homo2}) is surjective.  Suppose that $[\alpha]\in H^r(K, D)$, where $\alpha \in K^r$ such that $D\alpha=0$. Note that $d'\alpha $ is both $d'$-exact and $d$-closed. Thus by the $dd'$-lemma there exist $\beta \in K^{r-1}$ such that $d'\alpha =d'd\beta=d'D\beta$. This implies that $d'(\alpha-D\beta)=0$. However,  $(d+d')(\alpha-D\beta)=D(\alpha-D\beta)=0$. It follows that $d(\alpha-D\beta)=0$.  
Now use the bigrading on $K$ to decompose $\alpha-D\beta$ as follows.
\begin{equation}\label{finite-sum} \alpha-D\beta=\displaystyle\sum_{p+q=r} \gamma^{p,q},\, \, \gamma^{p,q}\in K^{p,q}.\end{equation} By our assumption on the boundedness of $K^{*,*}$, (\ref{finite-sum}) is indeed a finite sum. Since $\alpha-D\beta$ is both $d$ and $d'$ closed, it is easy to see that $d'\gamma^{p,q}= d \gamma^{p,q}=0$ for all $p+q=r$. It follows that \[\displaystyle \sum_{p+q=k}[\gamma^{p,q}]\in \displaystyle \bigoplus_{p,q}H^{p,q}(K_{d'}, d);\] moreover, its image under (\ref{iso-2}) is $[\alpha]$. This completes the proof.

 


\end{itemize}
\end{proof}

\section{Equivariant Dolbeault cohomology and the $\overline{\partial}_G\partial_G$-lemma}\label{eq-dolbeault}


  Throughout this section, suppose that there is a holomorphic action of a compact Lie group $K$ on a $2n$ dimensional complex manifold $M$. On the space of differential forms $\Omega(M)$ the exterior differential $d$ splits as
$ d= \overline{\partial}+\partial$.  Accordingly on the space of equivariant differential forms $\Omega_K(M)$ the operator
$1\otimes d$ splits as $1\otimes d=1\otimes \overline{\partial}+1\otimes \partial$. For brevity, we will also abbreviate $1\otimes \overline{\partial}$ to $\overline{\partial}$ and $1\otimes \partial$ to $\partial$.

Since the action of $K$ is holomorphic,  $\Omega^{p,q}(M)$ is a $K$-module for all $(p, q)$. Thus the space 
\begin{equation}\label{dolb-bigrading} \Omega_K^{p,q}(M):=\bigoplus_{p'+i=p,q'+i=q}(S^i\mathfrak{k}^*\otimes \Omega^{p'q'}(M))^K\end{equation}
is well defined for all $(p, q)$.

For all $\xi\in\mathfrak{k}$, denote by $\xi_M^{1,0}$ and $\xi_M^{0,1}$ respectively the $(1,0)$ and $(0,1)$ components of the vector field on $M$ induced by $\xi$. Then on the Cartan complex $\Omega_K(M)$ the operator $d'$ splits as $d'=d^{'1,0}+d^{'0,1}$, where
\[ (d^{'1,0}\alpha)(\xi)= \iota(\xi_M^{1,0})(\alpha(\xi)),\,(d^{'0,1}\alpha)(\xi)=\iota(\xi_M^{0,1})(\alpha(\xi)), \, \forall\, \alpha\in \Omega_K(M).\]

Thus the equivariant exterior differential $d_K$ splits as $d_K=\overline{\partial}_K+\partial_K$, where

\[\begin{split} &
\partial_{K}=\partial+d^{'1,0},\,\,\overline{\partial}_{K}=\overline{\partial}+d^{'0,1}.
\end{split}\]

It is straightforward to check that \[\overline{\partial}_K^2=0,\,\partial_K^2=0,\,\partial_K\overline{\partial}_K+\overline{\partial}_K\partial_K=0.\]
\begin{definition} The \emph{equivariant Dolbeault cohomology} of $M$, denoted by $H_{K}^{p,*}(M, \mathbb{C})$, is defined to be the cohomology of the differential complex $\{ \Omega_K^{p,*}(M), \overline{\partial}_K\}$.
\end{definition}
 

The following result was due to Lillywhite \cite{Lilly98} and Teleman \cite{T00}.

\begin{theorem}\label{eq-formality1}(\cite[Thm. 5.1]{Lilly03}) Consider the holomorphic action of a compact Lie group $K$ on a compact K\"ahler manifold $M$. Assume that the action is equivariantly formal. Then the following properties hold true.
\begin{itemize}
\item[a)]  There is an isomorphism of $(S^i(\mathfrak{k}^*))^K$-modules \[H^{p,q}_K(M,\mathbb{C})\cong \displaystyle \bigoplus_{p'+i=p, q'+i=q}(S^i(\mathfrak{k}^*))^K\otimes H^{p',q'}_{\overline{\partial}}(M).\]

\item[b)]
\[ \text{ker}\,\overline{\partial}_K \cap \text{im}\,\partial_K=\text{im}\,\overline{\partial}_K\cap\text{ker}\,\partial_K=\text{im}\,\overline{\partial}_K\partial_K.\] 
\end{itemize}

\end{theorem}

Now set $\Omega_{K,\partial_K}(M)=\Omega_K(M)\cap \text{ker}\partial_K$. Since $\overline{\partial}_K$ anti-commutes with $\partial_K$, we get a differential complex $\{\Omega_{K,\partial_K}^{p,*}(M), \overline{\partial}_K\}$.
The $q$-th cohomology of this differential complex will be denoted by $H(\Omega_{K,\partial_K}^{p,q}(M),\overline{\partial}_K)$.


As an immediate consequence of Theorem \ref{formality}, Proposition \ref{decomposition}, and Theorem \ref{eq-formality1}, we have the following result.
\begin{theorem}\label{eq-hodge-decom} \begin{itemize} For the holomorphic Hamiltonian action of a compact Lie group $K$ 
on a compact K\"ahler manifold $(M,\omega)$, the following properties hold true.

 \item[a)] The homomorphism
\begin{equation}\label{iso-1} H^{p,q}(\Omega_{K,\partial_K}(M),\overline{\partial}_K)\rightarrow H^{p,q}_K(M).\end{equation}
 induced by the inclusion $\{\Omega_{K,\partial_K}(M), \overline{\partial}_K\}\hookrightarrow \{\Omega_K(M), \overline{\partial}_K\}$ is an 
 isomorphism.
 \item[b)] The homomorphism
\begin{equation}\label{iso-2} \displaystyle \bigoplus_{p+q=r} H^{p,q}(\Omega_{K,\partial_K}(M),\overline{\partial}_K)\rightarrow H_K^r(M,\mathbb{C}).\end{equation}
 induced by the inclusion $\{\Omega_{K,\partial_K}(M), \overline{\partial}_K\}\hookrightarrow \{\Omega_K(M), d_K\}$ is an 
 isomorphism. Thus the data $(H^r_K(M,\mathbb{R}), H^{p,q}(\Omega_{K,\partial_K}(M))$ defines a (pure) real Hodge structure of weight $r$.
\end{itemize}
 \end{theorem}

\begin{definition} \label{eq-hodge-number} The dimension of the complex vector space $H^{p,q}_K(M,\mathbb{C})$, denoted by $h_K^{p,q}(M)$, is defined to be the \emph{equivariant Hodge number} of the $K$-manifold $M$.
\end{definition}

\section{Transversely K\"ahler foliations}\label{foliation}
Let $\mathcal{F}$ be a foliation on a smooth manifold $Z$. Throughout this paper we will denote by $\mathfrak{X}(\mathcal{F})$ the space of smooth vector fields which are tangent to the leaves of $\mathcal{F}$, and by $T\mathcal{F}$ the tangent bundle of the foliation.  We say that a vector field $X$ on $Z$ is \emph{foliate}, if $[X,Y]\in\mathfrak{X}(\mathcal{F})$, for all $Y\in\mathfrak{X}(\mathcal{F})$. We will denote by $\mathfrak{X}(Z,\mathcal{F})$ the space of foliate vector fields on the foliated manifold $(Z,\mathcal{F})$. Clearly we have that $\mathfrak{X}(\mathcal{F})\subset \mathfrak{X}(Z,\mathcal{F})$.  A \emph{transverse vector field} is an equivalence class in the quotient space $\mathfrak{X}(Z,\mathcal{F})/\mathfrak{X}(\mathcal{F})$.
The space of transverse vector fields, denoted by $\mathfrak{X}(Z/\mathcal{F})$,  is a Lie algebra with
a Lie bracket induced from that of $\mathfrak{X}(Z,\mathcal{F})$

The space of \emph{basic forms} on $Z$ is defined to be
\[
\Omega_{bas}(Z)=\bigl\{\alpha\in\Omega(Z)\,|\,\iota(X)\alpha=\mathcal{L}(X)\alpha=0,\,\forall\,X\in\mathfrak{X}(\mathcal{F})\bigr\}.
\]
Since the exterior differential operator $d$ preserves basic forms, we obtain a sub-complex  $\{\Omega^{*}_{bas}(Z),d\}$ of the de Rham complex, called the \emph{basic de Rham complex}.
The associated cohomology $H^{*}_B(Z)$
is called the \emph{basic cohomology}. Let $q$ be the codimension of the foliation $\mathcal{F}$.
If $H^{q}_B(Z)=\mathbb{R}$, we say that $\mathcal{F}$ is \emph{homologically orientable}.

Let $Q=TZ/T\mathcal{F}$ be the normal bundle of the foliation. A moment's consideration shows that for any foliate vector field $X$, and for any $(r, s)$-type tensor \[\sigma\in C^{\infty}(\underbrace{Q^*\otimes\cdots\otimes Q^*}_{r}\otimes\underbrace{Q\otimes\cdots \otimes Q}_s),\] the Lie derivative $\mathcal{L}_X\sigma$ is well defined.

\begin{definition} A \emph{transverse Riemannian metric} on a foliation $(Z,\mathcal{F})$ is a Riemannian metric $g$ on the normal bundle $Q$ of the foliation, such that $\mathcal{L}_Xg=0$,  for all $X\in\mathfrak{X}(\mathcal{F})$. We say that $\mathcal{F}$ is a Riemannian foliation if there exists a transverse Riemannian metric on $(Z,\mathcal{F})$.
\end{definition}
\begin{definition}
A  \emph{transverse almost complex structure} $\mathcal{J}$ on $(Z,\mathcal{F})$ is an almost complex structure $\mathcal{J}: TZ/T\mathcal{F}\rightarrow TZ/T\mathcal{F}$ such that
$\mathcal{L}_X \mathcal{J}=0$, for all $X\in \mathfrak{X}(\mathcal{F})$.
A  transverse almost complex structure $\mathcal{J}$ on $(Z,\mathcal{F})$ is said to be \emph{integrable}, if
 for all $p\in Z$, there exists an open neighborhood $U$ of $p$, such that for any two transverse vector fields $X$ and $Y$  on $U$ with respect to the foliation $\mathcal{F}\vert_U$, the Nijenhaus tensor
  $
  N_{\mathcal{J}}(X,Y)=[\mathcal{J}X,\mathcal{J}Y]-\mathcal{J}[\mathcal{J}X,Y]-\mathcal{J}[X, \mathcal{J}Y]-[X,Y]
  $
   vanishes. An integrable transverse almost complex structure is also called a \emph{transverse complex structure}.
The foliation $\mathcal{F}$ is said to be \emph{transversely holomorphic} if there is a transverse complex structure $\mathcal{J}$ on $(Z, \mathcal{F})$.\end{definition}

Throughout the rest of this section, assume that $\mathcal{F}$ is a foliation on a manifold $Z$ endowed with a transverse almost complex structure $\mathcal{J}$. To simplify notations, we will also denote by $\Omega_{ bas}(Z)$ and $\Omega_{hor}(Z)$ the space of complex basic differential forms and complex horizontal forms respectively. Now let $Q_{\mathbb{C}}$ be the complexification of the normal bundle $Q$ of the foliation.  Then $\mathcal{J}$ determines a decomposition of $\wedge^r Q_{\mathbb{C}}^*$ as follows.
\begin{equation}\label{basic-form-decom1}\wedge^r Q_{\mathbb{C}}^*=\displaystyle \bigoplus_{p+q=r} (\wedge^p Q^{*'})\otimes (\wedge^q Q^{*''}), \end{equation} where $Q^{*'}$ are $Q^{*''}$ are $\sqrt{-1}$ and $-\sqrt{-1}$-eigenbundle of $\mathcal{J}$ respectively. However, it is clear that there is  a natural isomorphism 
\begin{equation}\label{injection}\Omega_{hor}^r(Z) \rightarrow C^{\infty}(\wedge^r Q_{\mathbb{C}}^*).\end{equation}
Thus (\ref{basic-form-decom1} ) together with (\ref{injection}) induces the following decomposition of $\Omega^r_{hor}(Z)$.\begin{equation}\label{basic-form-decom2} \Omega^r_{hor}(Z) =\displaystyle \bigoplus_{p+q=r}\Omega^{p,q}_{hor}(Z).\end{equation}
\begin{definition}\label{dbar-operator}  Let $\alpha$ be a complex basic form of type $(p, q)$.  In view of the direct sum decomposition (\ref{basic-form-decom2}), define  $\overline{\partial}\alpha$ to be the $(p, q+1)$ component of $d\alpha$, and $\partial\alpha$ the $(p+1, q)$ component of $d\alpha$.
\end{definition} 

The proof of the following fact is analogous to the case of complex manifolds, and will be left as an exercise.
\begin{lemma} The transverse almost complex structure $\mathcal{J}$ is integrable if and only if $d=\overline{\partial} + \partial$.
In particular, when $\mathcal{J}$ is integrable, we have that 
\[\overline{\partial}^2=0,\,\,\overline{\partial}\partial+\partial \overline{\partial}=0, \,\,\partial^2=0.\]

\end{lemma}

\begin{definition} Assume that $\mathcal{F}$ is a transversely holomorphic foliation on $Z$. The \emph{basic Dolbeault cohomology} of $Z$, denoted by $H_{B,\overline{\partial}}^{p,*}(Z)$, is defined to be the cohomology of the differential complex
$\{\Omega_{bas}^{p,*}(Z),\overline{\partial}\}$.
The dimension of the complex vector space $H^{p,q}_{B,\overline{\partial}}(Z)$ is defined to be the basic Hodge number $h^{p,q}_B$ of the transversely holomorphic foliation $(Z, \mathcal{F})$.

 \end{definition}

\begin{definition}\label{kahler-foliation}
A \emph{transverse K\"{a}hler structure} on $(Z,\mathcal{F})$ consists of a transverse complex structure $\mathcal{J}$ and a transverse Riemannian metric $g$,  such that
 the tensor field $\omega$ defined by $\omega(X,Y)=g(X,\mathcal{J}Y)$ is anti-symmetric and closed when considered as a 2-form on $Z$ given by the injection $\bigwedge^{2}Q^*\rightarrow\bigwedge^{2}T^{*}Z$.
The $2$-form $\omega$  will be called a \emph{transverse K\"{a}hler form}.  $\mathcal{F}$ is said to be a transversely K\"ahler foliation if there exists a transverse K\"ahler structure on $(Z,\mathcal{J})$.
\end{definition}

The following result is due to El Kacimi \cite{KA90}.

\begin{theorem}\label{basic-ddbar-lemma} Suppose that $\mathcal{F}$ is a homologically orientable transversely K\"ahler foliation on a compact manifold $Z$. Then on the space of basic forms $\Omega_{bas}(Z)$ the following $\overline{\partial}\partial$-lemma holds.
\[\text{ker}\,\overline{\partial} \cap \text{im}\,\partial=\text{im}\, \overline{\partial}\cap \text{ker}\,\partial=\text{im}\,\overline{\partial}\partial.\]

\end{theorem}

\section{Kirwan map and the K\"ahler quotients}\label{hodge-structure}

Let $(M,\omega, \mathcal{J})$ be a K\"ahler manifold with a K\"ahler two form $\omega$ and a compatible integrable almost complex structure $\mathcal{J}$, and let $g(\cdot, \cdot):=\omega(\cdot,\mathcal{J}\cdot)$ be the associated K\"aher metric.
Assume that there is a Hamiltonian action of a Lie group $K$ (not necessarily compact) on $(M, \omega)$ with a moment map $\Phi: M\rightarrow \mathfrak{k}^*$, where $\mathfrak{k}^*$ is the dual space of  $\mathfrak{k}:=\text{ Lie}(K)$, that the action is also holomorphic, and that $0\in \mathfrak{k}^*$ is a regular value of the moment map. 

By assumption,  the level set $Z:=\Phi^{-1}(0)$ is an embedded submanifold of $M$ on which the action of $K$ is locally free. Thus the $K$-action generates a regular foliation $\mathcal{F}$ on $Z$. Let $V=T\mathcal{F}$, and for $\xi\in\mathfrak{k}$, 
let $\xi_Z$ be the fundamental vector field on $Z$ generated by $\xi\in \mathfrak{k}$. Then by definition, for all
$z\in Z$, $V_z=\text{span}\{\xi_{Z,z}\,\vert\, \xi\in \mathfrak{k}\}$. It follows easily from the Hamiltonian equation (\ref{Hamiltonian-eq}) that $V_z=\text{ker}\,(\omega\vert_Z)_z$. As a result, for $ X\in V_z$, and for $Y\in T_z Z$, we have that $g(Y, \mathcal{J}X)=\omega(Y, \mathcal{J}^2 X)=-\omega(Y, X)=0$.

 We have thus proved that $\mathcal{J}V$ is orthogonal to $TZ$ in $TM\vert_Z$. A simple dimension count shows that the subbundle $\mathcal{J}V$ is the orthogonal complement of $TZ$ in $TM\vert_Z$. Let $E$ be the orthogonal complement of $V$ in $T Z$. Then $E$ is the orthogonal complement of $W:=V\oplus \mathcal{J}V$ in $TM\vert_Z$. Since both $W$ and $g$ are invariant under the action of $\mathcal{J}$ and $K$, we obtain a $K$-invariant almost complex structure $\mathcal{J}$ on $E$.

Now let $T_{\mathbb{C}}M$ and $E_{\mathbb{C}}$ be the complexfication of $TM$ and $E$ respectively, and let $E^{1,0}_{\mathbb{C}}$ and $E^{0,1}_{\mathbb{C}}$ be the $\sqrt{-1}$-eigenbundle and $-\sqrt{-1}$-eigenbundle respectively of the almost complex structure $\mathcal{J}: E_{\mathbb{C}}\rightarrow E_{\mathbb{C}}$. We say that a complex tangent vector $A+\sqrt{-1}B$ is tangent to $Z$, where $A,B\in T_zM$, $z\in Z$, if both $A$ and $B$ lie in $T_zZ$.  We first make the following simple observations.
\begin{lemma} \label{observation} \begin{itemize} \item [a)] Suppose that $X\in T^{1,0}_{\mathbb{C},z}(M)$, where  $z\in Z$. Then $X$ is tangent to $Z$ if and only if $X\in E^{1,0}_{\mathbb{C},z}$.
\item[b)] If $X_1, X_2\in C^{\infty}(E_{\mathbb{C}}^{1,0})$, then 
$[X_1,X_2]\in C^{\infty}(E_{\mathbb{C}}^{1,0})$. \end{itemize}
\end{lemma}

\begin{proof}\begin{itemize}\item [a)] Suppose that $X\in T^{1,0}_{\mathbb{C},z}(M)$ is tangent to $Z$. Since $X$ is of $(1,0)$ type,  we can write $X$ as $X= X_1+(\xi_{Z,z}-\mathbf{i}\mathcal{J}\xi_{Z,z})$ for $X_1\in E^{1,0}_{\mathbb{C},z}$ and $\xi\in \mathfrak{k}$. Since $X$, $X_1$ and $\xi_{Z,z}$ are all tangent to $Z$,  we must have that $<X, d\Phi^{\xi}>=<X_1, d\Phi^{\xi}>=<\xi_{Z,z}, d\Phi^{\xi}>=0$. It follows that
\[ 0=<\mathbf{i}\mathcal{J}\xi_{Z,z}, d\Phi^{\xi}>=-\mathbf{i}\omega(\xi_{Z,z}, \mathcal{J}\xi_{Z,z})=-\mathbf{i}g(\xi_{Z,z},\xi_{Z,z}).\]
So $\xi_{Z,z}=0$, and $X=X_1\in E^{1,0}_{\mathbb{C},z}$. The other direction is obvious. 

\item[b)] Extend $X_i$ to a vector field $\widetilde{X}_i$ on an open neighborhood $U$ of $Z$, $i=1, 2$.  Since $\mathcal{J}$ is an integrable almost complex structure, it follows from the vanishing of $N_\mathcal{J}(\tilde{X}_1, \tilde{X}_2)$  that $\mathcal{J}[X_1,X_2]=\sqrt{-1}[X_1,X_2]$ on $Z$.  Thus $[X_1, X_2]$ is a vector field of type $(0,1)$. 
Since $Z$ is a submanifold of $M$, and since both $X_1$ and $X_2$ are tangent to $Z$, we must have that $[X_1,X_2]$ is tangent to $Z$. Therefore by Part a) of Lemma \ref{observation},  $[X_1,X_2]$ must be a section of $E^{1,0}_{\mathbb{C}}$.

\end{itemize}

\end{proof}

\begin{proposition}\label{transverse-kahler-foliation} The foliation $\mathcal{F}$ induced by the action of $K$ on $Z=\Phi^{-1}(0)$ is transversely K\"ahler. Moreover, if we assume that the moment map is proper, and that the Lie group $K$ is compact, then the foliation $\mathcal{F}$ is also homologically orientable.
\end{proposition}

\begin{proof} Let $Q=TZ/T\mathcal{F}$ be the normal bundle of the foliation, and let $\psi: E\rightarrow Q$ be the restriction to $E$ of the projection map $proj: TZ \rightarrow TZ/T\mathcal{F}$. Clearly $\psi$ is an isomorphism from $E$ to $Q$. Define 
$h=(\psi^{-1})^{*}g\vert_E$, and define $\mathcal{J}: Q\rightarrow Q$ to be the unique bundle map on $Q$ such that 
\[ \mathcal{J} \psi(X)= \psi \mathcal{J}  (X), \,\,\forall\, X\in C^{\infty}(E).\]
Since $g\vert_E$ and $\mathcal{J}: E\rightarrow E$ are $K$-invariant, it is straightforward to check that $h$ is a transverse Riemannian metric, and that $\mathcal{J}: Q\rightarrow Q$ is a transverse almost complex structure. The integrability of $\mathcal{J}:Q\rightarrow Q$ follows easily from Part b) of Lemma \ref{observation}. 


Now consider the two tensor $\sigma$ on $Q$ given by $\sigma(\cdot, \cdot)=h(\cdot, \mathcal{J}\cdot)$. It is anti-symmetric since $g$ is compatible with $\mathcal{J}$. Moreover, under the injection $\wedge^2 Q^*\rightarrow \wedge^2 T^*Z$ $\sigma$ gets mapped to the closed $2$-form $\omega\vert_Z$.  This proves that $\mathcal{F}$ is transversely K\"ahler. 


Finally, assume that the moment map $\Phi$ is proper, and that the Lie group $K$ is compact. Let $2q=\text{codim}(\mathcal{F})$. Under the assumption, the top basic cohomology $H^{2q}_{\mathbb{R}}(Z,\mathcal{F})$ of $\mathcal{F}$ is naturally isomorphic to the top de Rham cohomology of the quotient space $Z/K$, which is a compact symplectic orbifold. Thus $H^{2q}_{\mathbb{R}}(Z,\mathcal{F})\cong  \mathbb{R}$. This completes the proof of Proposition \ref{transverse-kahler-foliation}. 
\end{proof} 
\begin{remark}\label{homological-orientable}  The claims of Proposition \ref{transverse-kahler-foliation} continue to hold
without assuming that $K$ is compact.  However, the proof would involve the notion of Molino's sheaf,  and more generally, his structure theory for a Riemannian foliation. We refer the interested readers to \cite[Prop. A1]{LY19} for a rigorous proof, and to \cite{Mo88} and \cite{LS18} for a detailed exposition on Molino sheaf and Molino's structure theory of Riemannian foliations.\end{remark}

Let $\Omega_{bas,\partial}=\Omega_{bas}(Z)\cap \text{ker}\, \partial$, and let $H^{p,q}(\Omega_{bas,\partial}, \overline{\partial})$ be the cohomologies associated to the differential complex $\{\Omega^{p,*}_{bas,\partial},\overline{\partial}\}$.
The following result is an easy consequence of Proposition \ref{decomposition}, Theorem \ref{basic-ddbar-lemma}, and Proposition \ref{transverse-kahler-foliation}. 
 %

\begin{corollary}\label{basic-hodge-decom}  Assume that the moment map $\Phi$ is proper, and that the Lie group $K$ is compact. Then we have that
\begin{equation}\label{quasi-iso}  H^{p,q}_{B,\overline{\partial}}(Z)=H^{p,q}(\Omega_{bas,\partial},\overline{\partial}) 
\end{equation}
\begin{equation}\label{basic-hodge-decom} H^r_B(Z) = \displaystyle \bigoplus_{p+q=r} H^{p,q}(\Omega_{bas,\partial},\overline{\partial}) .\end{equation}
In particular, the data $(H^r_B(Z,\mathbb{R}), H^{p,q}(\Omega_{bas,\partial},\overline{\partial}))$ defines a (pure) real Hodge structure of weight $r$.

\end{corollary}


 
 
 Next let $\xi_1, \cdots, \xi_k$ be a basis of $\mathfrak{k}$, and let $\theta^1,\cdots, \theta^k$ be connection $1$-forms on $Z$ determined by the equations \begin{equation}\label{connection-elements} \iota(X)\theta^i=0,\,\forall\, X\in C^{\infty}(E),\,\forall\, 1\leq i\leq k,\,\text{and}\,\iota(\xi_j)\theta^i=\delta_i^j, \,\forall\,1\leq i,j\leq k.\end{equation}

 Then for all $1\leq l\leq k$, we have a curvature two form $\mu^l \in \Omega_{hor}(Z)$ as given in (\ref{curvature-2-form}). Moreover, we also have a well defined projection operator 
 $Hor: \Omega(Z)\rightarrow \Omega_{hor}(Z)$ as given in (\ref{hor-proj}).
 The following lemma is a crucial step towards establishing the main result of this paper.
  
 \begin{lemma} \label{type-1-1-curvature} \begin{itemize}\item [a)] Suppose that $\alpha\in \Omega^{p,q}(M)$. Then 
 $Hor(i^*\alpha)\in\Omega_{hor}^{p, q}(Z)$ is a horizontal form of type $(p,q)$ relative to the direct sum decomposition introduced in (\ref{basic-form-decom2}).
 Here $i^*$ is the pullback on differential forms induced by the inclusion map $i: Z\hookrightarrow M$.
 \item[b)] For all $1\leq l\leq k$, $\mu^l$ is a horizontal form of type $(1,1)$ relative to the bi-grading introduced in
 (\ref{basic-form-decom2}).\end{itemize}
  \end{lemma}
  
  \begin{proof}  \begin{itemize}
  \item[a)]  To show that $Hor(i^*\alpha)$ is of type $(p, q)$, it suffices to show that for any vectors $X_1,\cdots X_r\in E^{1,0}_{\mathbb{C},z}$, and any vectors $Y_1,\cdots, Y_s\in E^{0,1}_{\mathbb{C},z}$, where $r+s=p+q$ and $z\in Z$,  we have that
  \[ Hor(i^*\alpha)(X_1,\cdots, X_r,Y_1,\cdots, Y_s)=0,\]
  provided $r\neq p$.  By Theorem \ref{horizontal-decom}, $i^*\alpha$ admits a unique expression
  \begin{equation}\label{comparison}  i^*\alpha= \displaystyle \sum_{ I}\theta^Ih_{I}+Hor(i^*\alpha),\,\,h_{I}\in \Omega_{hor}(Z).\end{equation} Here in the above summation $\theta^I=(\theta^1)^{i_1}\cdots(\theta^k)^{i_k}$ for some non-negative integers $i_1,\cdots i_k$ satisfying $i_1+\cdots+i_k\geq 1$.  Therefore (\ref{comparison}), together with the first half of (\ref{connection-elements}), implies that
 \[ Hor(i^*\alpha)(X_1,\cdots,X_r,Y_1,\cdots,Y_s)= i^*\alpha(X_1,\cdots,X_r,Y_1,\cdots,Y_s).\]
  However, it follows easily from the definition of the almost complex structure $\mathcal{J}$ on $E$ that  $i_*(X_i)\in T_{\mathbb{C},z}^{1,0}(M)$, and that $i_*(Y_j)\in T_{\mathbb{C},z}^{0,1}(M)$, $1\leq i\leq r$, $1\leq j\leq s$. Since $\alpha\in\Omega^{p,q}(M)$,  we must have that $\alpha(i_*(X_1),\cdots,i_*(X_r),i_*(Y_1),\cdots,i_*(Y_s))=0$ provided $r\neq p$.

  \item[b)]Let $X_1, X_2\in  C^{\infty}(E^{1,0}_{\mathbb{C}})$. Since $\mu^l$ is a real two form, to show Part b) of Lemma \ref{type-1-1-curvature} it suffices to show that $ \mu^l(X_1,X_2)=0$. By definition,  
  \[ \begin{split}  \mu^l(X_1,X_2)&= (d\theta^l+c_{ij}^l\theta^i\wedge \theta^j)(X_1,X_2)\\
  &= d\theta^l(X_1,X_2)\\&
  =X_1\left(\theta^l(X_2)\right)-X_2\left(\theta^l(X_1)\right) -\theta^l([X_1,X_2])
  \\&=0\end{split}\]
  Here we have used Part b) of Lemma \ref{observation} to show that the last equality holds.
  \end{itemize}
  
 \end{proof}

 \begin{definition} \label{kirwan-forms} Let $\mathcal{C}: \Omega_K(Z)\rightarrow \Omega_{bas}(Z)$ be the Cartan operator as introduced in Definition \ref{cartan-operator}, and let $i:Z\rightarrow M$ be the inclusion map. We define
\[ \kappa:= \mathcal{C}\circ i^*: \Omega_K(M)\rightarrow \Omega_{bas}(Z)\]
to be \emph{the Kirwan map at the level of differential forms}. 
\end{definition}

\begin{remark}  It is clear from Remark \ref{dependence} that the Kirwan map introduced in Definition \ref{kirwan-forms}
depends only on the K\"ahler metric $g$.
\end{remark}

 \begin{proposition}\label{hodge-morphism} \begin{itemize}\item[a)] $\forall\,\alpha\in\Omega_K^{p,q}(M)$, $\kappa \alpha \in \Omega_{bas}^{p,q}(Z)$. \item[b)] $\forall\,\alpha\in \Omega^{p,q}_K(M)$,
 \[\kappa\overline{\partial}_K\alpha=\overline{\partial}\kappa \alpha.\]
\end{itemize}

 \end{proposition}
 
\begin{proof} Part a) of Proposition \ref{hodge-morphism} is an immediate consequence of Lemma \ref{type-1-1-curvature}.
To show that Part b) holds, first note that by Theorem \ref{cartan-homotopy} 
 $ \kappa d_K\alpha=d\kappa \alpha$, $\forall\, \alpha\in \Omega_{K}^{p,q}(M)$. Since $d_K=\partial_K+\overline{\partial}_K$ and $d=\partial+\overline{\partial}$, we have that
 \[\kappa \overline{\partial}_K\alpha+\kappa\partial_K\alpha=\overline{\partial}\kappa\alpha+\partial\kappa \alpha.\]
However, by Part a) of Proposition \ref{hodge-morphism} the Kirwan map $\kappa$ respects the bi-gradings.  
By comparing the types of the forms we see that \[ \kappa\overline{\partial}_K\alpha=\overline{\partial}\kappa \alpha.\]

\end{proof}

 From Proposition \ref{hodge-morphism} we obtain 
two chain maps 
\begin{equation} \label{dolbeault-chain-map}\kappa: \{\Omega^{p,*}_{K}(M), \overline{\partial}_K\}\rightarrow \{\Omega_{bas}^{p,*}(Z),\overline{\partial}\},\end{equation}
\begin{equation}\label{dolbeault-chain-map2} \kappa: \{\Omega_{K,\partial_K}^{p,*}(M), \overline{\partial}_K\} \rightarrow \{\Omega_{bas,\partial}^{p, *}(Z), \overline{\partial}\}.\end{equation}

Bu abuse of notations we will also denote by $\kappa: H^{p,q}_K(M)\rightarrow H_{\overline{\partial}}^{p,q}(Z)$ and $\kappa: H^{p,q}(\Omega_{K,\partial_K}(M), \overline{\partial}_K)\rightarrow H^{p,q}(\Omega_{bas,\partial}(Z), \overline{\partial})$
 the homomorphisms 
induced by the chain maps (\ref{dolbeault-chain-map}) and  (\ref{dolbeault-chain-map2}) respectively, and will call them Kirwan maps as well.  Note that by Theorem \ref{cartan-homotopy}, $\kappa: \{\Omega_K(M), d_K\}\rightarrow \{\Omega_{bas}(Z),d\}$ is also a chain map. It induces a homomorphism $\kappa :H_K(M)\rightarrow H_B(Z)$, which is the usual Kirwan map at the level of cohomologies. We are ready to prove the following theorem.

\begin{theorem}\label{Kirwan-surj-eq-dolbeault} Assume that the equivariant K\"ahler $K$-manifold $M$ is compact, and that the Lie group $K$ is compact. Then the Kirwan map
\[ \kappa: H^r_K(M,\mathbb{R}) \rightarrow H_{B}^r(Z,\mathbb{R}) \]
is a morphism of real Hodge structures of bi-degree $(0,0)$ with respect to the pure Hodge structures described in Theorem \ref{eq-hodge-decom} and Corollary \ref{basic-hodge-decom} respectively.
As a result, the Kirwan map
\[ \kappa: H_K^{p,q}(M)\rightarrow H^{p,q}_{B,\overline{\partial}}(Z)\]
is surjective, and so the following inequality holds.
\begin{equation}\label{hodge-number-inequality} h^{p,q}_B(M)\leq h_K^{p,q}(M)\end{equation}
\end{theorem}

\begin{proof}  The first claim follows easily from Proposition \ref{hodge-morphism}. This implies that
we have the following commutative diagram.
\[ \begin{tikzcd}
\displaystyle \bigoplus_{p+q=k}H^{p,q}(\Omega_{K,\partial_K}(M),\overline{\partial}_K)\arrow[r, "\cong"] \arrow[d,"\kappa"]& H^k_K(M)\arrow[d, "\kappa"]\\
\displaystyle \bigoplus_{p+q=k}H^{p,q}(\Omega_{bas,\partial}(Z), \overline{\partial})\arrow[r, "\cong"]& H^k_B(Z)
\end{tikzcd} \]
By Theorem \ref{eq-hodge-decom} and Corollary \ref{basic-hodge-decom}, both the top and the bottom horizontal maps in the above diagram are isomorphisms. By Theorem \ref{Kirwan-surj}, the right vertical map is surjective. It follows that the left vertical map must be surjective as well. However, since the left vertical map maps each component $H^{p,q}(\Omega_{K,\partial_K}(M),\overline{\partial}_K)$ into $H^{p,q}(\Omega_{bas,\partial},\overline{\partial})$,  the map
\[ \kappa: H^{p,q}(\Omega_{K,\partial_K}(M), \overline{\partial}_K)\rightarrow H^{p,q}(\Omega_{bas,\partial}(Z), \overline{\partial})\]
must be surjective as well for each pair of $(p, q)$.
Now consider the commutative diagram
\[\begin{tikzcd}
 H^{p,q}(\Omega_{K,\partial_K}(M),\overline{\partial}_K)\arrow[r, "\cong"] \arrow[d, "\kappa"]& H^{p,q}_K(M)\arrow[d, "\kappa"]\\
H^{p,q}(\Omega_{bas,\partial}(Z), \overline{\partial})\arrow[r,"\cong"]& H^{p,q}_{B,\overline{\partial}}(Z)
\end{tikzcd} \]

Since the $K$-manifold $M$ is equivariantly formal, it follows from Theorem \ref{eq-hodge-decom} and Corollary \ref{basic-hodge-decom} that both the top and the bottom horizontal maps are isomorphisms. By our work above, the left vertical map is surjective. It follows that the right vertical map must be surjective as well.
\end{proof}

By our assumption, the quotient space $M_0=Z/K$ is an orbifold. Clearly,  the transverse K\"ahler structure on $Z$ naturally descends to a K\"ahler structure on $M_0$. We note that there is a natural isomorphism from the complex of basic forms $\{\Omega_{bas}(Z), d\}$ to the complex of differential forms $\{\Omega(M_0),d\}$, and that this natural isomorphism respects the type of differential forms. Thus we must have that $H_{B,\overline{\partial}}^{p,q}(Z)\cong H_{\overline{\partial}}^{p,q}(M_0)$, which implies that $h_B^{p,q}(Z)=h^{p,q}(M_0)$.
The following result is an immediate consequence of Theorem \ref{Kirwan-surj-eq-dolbeault}, which also answers the open question raised by Weitsman.

\begin{corollary}\label{hodge-diamond}  Consider the holomorphic Hamiltonian action of a compact Lie group $K$ on a compact K\"ahler manifold $(M,\omega)$ with a moment map $\Phi: M\rightarrow \mathfrak{k}^*$. Assume that $K$ acts locally freely on the level set $Z=\Phi^{-1}(0)$, and that $M_0=Z/K$ is the K\"ahler quotient. Then there is a natural surjective Kirwan map

\[ \kappa: H_K^{p,q}(M)\rightarrow H^{p,q}_{\overline{\partial}}(M_0).\]

In particular, if $h^{p,q}(M)=0$ when $p\neq q$, then 
$h^{p,q}(M_0)=0$ when $p\neq q$.
\end{corollary}



We would like to emphasize that the definition of the Kirwan map at the level of differential forms does not depend on the compactness of $M$. As a result, the method developed in this paper may be applied to situations where $M$ is not compact as well. For example, according to the work \cite{LT97} of Lerman and Tolman, any complete simplicial toric variety can be obtained as a symplectic quotient of a complex vector space $\mathbb{C}^N$ by the linear action of a compact abelian group $K$.  Since $H_K(\mathbb{C}^N)\cong \mathbb{C}[x_1,\cdots, x_k]$, and since the Kirwan surjectivity theorem is known to hold in this case, one can easily derive from Proposition \ref{hodge-morphism} a symplectic proof of the following well-known result on the Hodge numbers of a complete simplicial toric variety.

\begin{theorem}\label{simplicial-toric}(\cite[Thm. 9.3.2]{CLS11}) For a complete simplicial toric variety $X$, the Hodge number $h^{p,q}(X)=0$ provided $p\neq q$.

\end{theorem}

Finally we would like to comment that the techniques developed in this paper may well be used to attack problems that can not be accessed using methods from algebraic geometry. For instance, given an arbitrary irrational simple polytope, Prato \cite{Pra01} constructed a foliation analogue of symplectic toric manifolds, called a symplectic toric quasifold, using a variation of the Delzant's symplectic quotient construction. Prato's symplectic toric quasifolds do not exist in the algebraic category. However, Battaglia and Zaffran \cite{BZ15} observed that they can be realized as leaf spaces of holomorphic foliations, and conjectured that their basic Hodge numbers satisfy $h^{p, q}=0$ if $p\neq q$.  We believe that the conjecture of Battaglia and Zaffran can be confirmed using an argument similar to the one that we outlined above to show Theorem \ref{simplicial-toric}.  However,  the reason why the Kirwan surjectivity continues to hold in this situation requires a careful explanation. We will leave all the details in a forthcoming research paper.


\end{document}